\documentclass[12pt]{article}
 \usepackage{}
 \usepackage{amssymb}
\usepackage{amsfonts}
\usepackage{mathrsfs}
\usepackage{mathrsfs}
  \usepackage{amsthm}
 \usepackage{amsmath}

\def\sqr#1#2{{\vcenter{\vbox{\hrule height.#2pt
              \hbox{\vrule width.#2pt height#1pt \kern#1pt \vrule width.#2pt}
              \hrule height.#2pt}}}}
\def\a{\alpha}
\def\b{\beta}

\def\e{\varepsilon}

\def\l{\lambda}

\def\mf{\mathcal {F}}
\def\me{\mathbb{E}}

\def\R{{\bf R}}

\def\D{\Delta}

\def\L{\Lambda}

\def\O{\Omega}

\def\no{\noindent}

\def\ms{\medskip}
\def\bs{\bigskip}
\def\q{\quad}
\def\qq{\qquad}

\def\lan{\mathop{\langle}}
\def\ran{\mathop{\rangle}}

\def\max{\mathop{\rm max}}

\def\sup{\mathop{\rm sup}}

\def\ae{\hbox{\rm a.e.{ }}}
\def\as{\hbox{\rm a.s.{ }}}

\def\({\Big (}
\def\){\Big )}
\def\[{\Big[}
\def\]{\Big]}

\def\ds{\displaystyle}

\def\square#1{\vbox{\hrule\hbox{\vrule height#1%
     \kern#1\vrule}\hrule}}
\def\rectangle#1#2{\vbox{\hrule\hbox{\vrule height#1%
     \kern#2\vrule}\hrule}}

\font\tenbb=msbm10 \font\sevenbb=msbm7 \font\fivebb=msbm5

\newfam\bbfam
\scriptscriptfont\bbfam=\fivebb \textfont\bbfam=\tenbb
\scriptfont\bbfam=\sevenbb

\newtheorem{lemma}{Lemma}[section]
\newtheorem{remark}{Remark}[section]

\newtheorem{theorem}{Theorem}[section]

\newtheorem{definition}{Definition}[section]

\makeatletter
   
   \@addtoreset{equation}{section}
\makeatother

\begin{document}
\title{\bf A Semidiscrete Galerkin Scheme for Backward Stochastic Parabolic Differential Equations}

\author{Yanqing Wang\thanks{School of Mathematics and Statistics, Southwest University, Chongqing 400715, P.~R.~China. Part of this work was finished when the author was a Ph.D. student at the "Key Laboratory of Systems and Control, Academy of Mathematics and Systems Science, Chinese Academy of Sciences€. The author is supported by the National Basic Research Program of China (973 Program) under grant 2011CB808002, by the NSF of China under grants 11231007 and 11101452, and by Fundamental Research Funds for the Central Universities under grants SWU113038 and XDJK2014C076.  {\small\it E-mail:} {\small\tt
yqwang@amss.ac.cn}.\ms}}


\maketitle

\begin{abstract}
In this paper, we present a numerical
scheme to solve the initial-boundary value problem for backward stochastic
partial differential equations of parabolic type. Based on the Galerkin method, we approximate the original equation by a family of backward stochastic differential equations (BSDEs, for short), and then solve these BSDEs by the time discretization. Combining the truncation with respect to the spatial variable and the backward Euler method on time variable, we obtain the global $L^2$ error estimate.

 \no

\end{abstract}

\bs

\no \bs

\no{\bf Key Words}.  backward stochastic parabolic differential
equation, backward stochastic differential equation, Galerkin method, strong convergence.

\bs \no{\bf 2000 Mathematics Subject Classification}. 60H15, 65M60.

\section{Introduction}


Let $T\in(0,+\infty)$, $(\O,\mf,\mathbb{F},P)$ be a complete probability space and
$\mathbb{F}=\{\mf_t,t\in [0,T]\}$ be the natural filtration generalized by a 1-dimensional Wiener process  $\{W(t):t\in [0,T]\}$ satisfying the usual conditions.
The purpose of this work is to present a numerical scheme for solving the following backward stochastic parabolic differential equation (BSPDE, for short):
\begin{equation}\label{110}
\left\{
\begin{array}{lll}
dq(t,x)=(-\Delta q(t,x)+f(t,x,q(t,x),r(t,x)))dt+r(t,x)dW(t), & \mbox{in}\,\, [0,T)\times D, \\
q(t,x)=0,      & \mbox{on}\,\, [0,T)\times \partial{D}, \\
q(T,x)=q_T(x), & \mbox{in}\,\, D. \\
\end{array}
\right.
\end{equation}
Here, $D\subset \R^d$ is a bounded domain with $C^2$ boundary ($d\in\bf N$), $\D=\sum_{k=1}^d \partial^2/\partial x_k^2$ is the Laplacian, and $f$ and $q_T$ are given data satisfying suitable conditions to be given later.

BSPDEs are nontrivial extensions of BSDEs, which
 possess
interesting theoretical values and were originally introduced in the study of optimal control problems as the adjoint equation
appeared in the Pontryagin
maximum principle when the controlled system is a  stochastic parabolic differential equation
(see, e.g., \cite{Hu-Peng90,Lu-Zhang}), then are proved usefully in nonlinear filtering \cite{Pardoux79},
mathematical finance \cite{Hu06} and so on.

The well-posedness of general BSPDEs driven by Wiener process has been
considered in a number of papers using either duality
techniques \cite{Lu-Zhang,Zhou92} or martingale representation and fixed point arguments
\cite{Hu-Peng91,Tessitore}.

In recent years, the study of numerical solutions to stochastic differential equations becomes an active topic, and has attracted considerable attention in many fields such as control theory and mathematical finance.

Up to now, some numerical schemes for the forward stochastic partial differential equations have been presented:  the Galerkin approximation (\cite{Grecksch-Kloeden96}), the finite difference method in space and time (see, e.g., \cite{Gyongy98,Shardlow99}), the finite element method (see, e.g., \cite{Walsh05, Yan04}), the stochastic Taylor expansion method (see, e.g., \cite{Jentzen-Kloeden09}), the Wiener chaos expansion method (see, e.g., \cite{Hou06}) and so on.


On the other hand, there are also several algorithms for solving BSDEs. First of all, based on the four step scheme \cite{Ma-Protter-Yong94}, and using the relations between BSDEs and PDEs, Douglas et al. \cite{Douglas-Ma-Protter96}, Milstein and Tretyakov \cite{Milstein-Tretyakov06} obtained a numerical algorithm to solve BSDEs. The second kind of algorithms is the backward Euler method (see \cite{Bender-Denk07,Bouchard-Touzi04,Gobet-Lemor-Warin,Hu-Nualart-Song,Zhang04} for more details), for which the most difficult part is to calculate the conditional expectation. The third one is the random walk approach, in which the Brownian motion is replaced by a scaled random walk (see \cite{Ma-Protter-Martin-Torres02} and the references therein). We should also mention works by Wang and Zhang \cite{Wang-Zhang11} presenting the finite transposition method, Lototsky et al. \cite{Lototsky-Mikulevicius-Rozovskii97}, Briand and Labart \cite{Briand-Labart12} giving the Wiener chaos expansion method and so on.

Compared with the development of numerical methods for BSDEs and forward stochastic partial differential equations, the study of numerical schemes for BSPDEs is quite limited. To the author's best knowledge, there exists no published work in this direction. Here we should mention the work of Yannacopoulos et al. \cite{Karatzas}, which only listed an idea on the numerical scheme for solving BSPDEs. Their method
depends on the Wiener chaos expansion, however they did not prove the convergence speed.

The aim of this work is to provide a numerical scheme for solving Eq. \eqref{110}
based on the semidiscrete Galerkin approximation. This scheme is divided into
two steps. Firstly, we approximate Eq. \eqref{110} by a family of BSDEs; then, adopting the backward Euler method, we give numerical solutions to the related
BSDEs. This work is an improved and complete version of \cite[Chapter 5]{Wang13}.

Compared to the Wiener chaos expansion method listed in \cite{Karatzas}, it seems that the Galerkin scheme is easier to operate, at least for some special cases. Indeed, if the
domain is ``good" enough and the eigenvalues, eigenfunctions are easy
to be computed, we can take the finite-dimensional approximation spaces to be the ones
spanned by suitable eigenfunctions.

The rest of the paper is organized as follows: In Section 2, we introduce our general setting and review the well-posedness of Eq. \eqref{110}. In Section 3, we make use of the space-discretised Galerkin approximation to Eq.\ \eqref{110} and construct the desired finite-dimensional spaces; then Eq. \eqref{110} is approximated by a family of BSDEs. We prove the strong convergence in appropriate spaces and obtain the rate of the convergence in the $L^2$ norm with respect to the space variable.
In Section 4,  we adopt the backward Euler method, for which the convergence and error analysis are also provided.


\section{Preliminaries}

Let $H$ be a Hilbert space with norm $\|\cdot\|_{H}$ and inner product $\lan \cdot,\cdot \ran_H$. The following classes of processes will be frequently used throughout this paper:
\begin{itemize}
\item For any $t\in[0,T]$, $L^2_{\mf_t}(\O; H)$ is the space of all $\mf_t$-measurable, $H$-valued
  random variables $\xi$ satisfying $\|\xi\|^2_{L^2_{\mf_t}(\O; H)}=\me\|\xi\|_{H}^2<\infty$;

\item $C_{\mathbb{F}}([0,T];L^2(\O; H))$ is the space of all $\mathbb{F}$-adapted, continuous, $H$-valued stochastic processes
  $X$ satisfying $\|X\|^2_{C_{\mathbb{F}}([0,T];L^2(\O; H))}=\sup_{t\in[0,T]}\me\|X(t)\|_H^2<\infty$;  $L^2_{\mathbb{F}}(\O;C([0,T]; H))$ is the subspace of $C_{\mathbb{F}}([0,T];L^2(\O; H))$ such that $\me(\sup_{t\in[0,T]}\|X(t)\|_{H}^2)<\infty$;

\item $L^2_{\mathbb{F}}(\O\times(0,T); H)$ is the space of all $\mathbb{F}$-adapted, $H$-valued stochastic processes $Y$
  satisfying $\|Y\|^2_{L^2_{\mathbb{F}}(\O\times(0,T); H)}=\me(\int_{0}^{T}\|Y(t)\|_{H}^2dt)<\infty$;

\item $\mathbb{D}^{k,p}(H)\,(k,\,p\geq 1)$ is the space of all ${\cal F}_T$-measurable, $H^{\otimes k}$-valued, $k$-times Malliavin differentiable random variables $\xi$ satisfying
$$\|\xi\|_{k,p}:=\Big(\me(|\xi|^p)+\sum_{j=1}^k\me(\|D^j\xi\|^p_{H^{\otimes j}})\Big)^{1/p}<+\infty;$$

\item $M^{2,p}(p\geq 2)$ is the space of all ${\cal F}_T$-measurable square integrable random variables $\xi$ which admits a stochastic integral representation:
$$\xi=\me\xi+\int_0^Tu(t)dW(t),$$
where $u(\cdot)$ is a progresively measurable process
satisfying
$\sup_{0\leq t\leq T}\me|u(t)|^p<+\infty;$

\item $\mathbb{L}_a^{1,2}(H)$ is the space of all $H$-valued progressively measurable processes
$\{u(t)\}_{t\in [0,T]}$ satisfying
\begin{description}
\item[(i)] For almost all $t\in[0,T]$, $u(t)\in \mathbb{D}^{1,2}(H)$;
\item[(ii)] $\me(\int_0^T |u(t)|^2dt+\int_0^T\int_0^T|D_{\theta}u(t)|^2d\theta dt)<\infty.$
\end{description}

\end{itemize}
Besides, denote by $\lan\cdot,\, \cdot\ran$ and $|\cdot|$ respectively the inner product and norm in different Euclidean spaces, which can be identified from the context; by $A^*$ the transport matrix of $A$.

Let us introduce the following two assumptions:

\textbf{(A1)}\q\it $f:\R^+\times D \times \R^d\times\R^d \longrightarrow \R^d$ is $\frac{1}{2}$-H\"{o}lder continuous with respect to $t$, i.e., there exists a positive constant $L$, such that $|f(t_1,x,y,z)-f(t_2,x,y,z)|\leq L\sqrt{|t_1-t_2|}$, for any $t_1,\,t_2\in \R^+,\,x\in D,\,y,z\in\R^d$; and has continuous and uniformly bounded first and second partial derivatives with respect to $y$ and $z$ (still denote this bound by $L$).
Moreover, $f(\cdot,\cdot,0,0)\in L^2(0,T;H_0^1(D))$. \rm

\textbf{(A2)}\q \it $q_T\in L^2_{\mf_T}(\O;H_0^1(D))\cap  \mathbb{D}^{2,q}(L^2(D)),$
$$\me|D_\theta q_T-D_{\theta'}q_T|\leq M|\theta-\theta'|,$$
$$\max\Big\{\me|q_T|_{L^2(D)}^q,\,\sup_{\theta\in[0,T]}\me|D_{\theta}q_{T}|^q,\,
\sup_{\theta'\in[0,T]}\sup_{\theta\in[0,T]}\me|D_{\theta'}D_{\theta}q_{T}|^q\Big\}\leq M,$$
where $q>4$ and $M$ is a positive constant. \rm\\

Define $A:\,D(A)=H^2(D)\cap H_0^1(D)\longrightarrow L^2(D)$ by $Af=\Delta f$, for all $f\in D(A)$. It is easy to show that $A$ is the infinitesimal generator of a $C_0$-semigroup
$\{e^{At}\}_{t\geq 0}$ on $L^2(D)$. Then Eq. (\ref{110}) can be re-writen as the following abstract form:
\begin{equation}\label{111}
\left\{
\begin{split}
dq(t)&=\big(-A q(t)+f(t,\cdot,q(t),r(t))\big)dt+r(t)dW(t),\q   t\in [0,T), \\
 \displaystyle
 q(T)&=q_T.
\end{split}
\right.
\end{equation}

Now, we recall the definition of solution to Eq. (\ref{111})
and its well-posedness. We refer the reader to \cite{Hu-Peng91,Tessitore} for details.

\begin{definition}\label{definition of mild solution}%
A pair of random fields $(q,r)\in C_{\mathbb{F}}([0,T];L^2(\O;L^2(D)))$ $\times
L^2_{\mathbb{F}}(\O\times(0,T);L^2(D))$ is called a mild solution to Eq. \eqref{111}, if
for every $t\in [0,T]$, it holds
that
\begin{equation*}
\begin{aligned}
q(t)+\int_t^Te^{A(s-t)}f(s,\cdot, q(s),r(s))ds+\int_t^Te^{A(s-t)}r(s)dW(s)=e^{A(T-t)}q_T,\,\as
\end{aligned}
\end{equation*}
\end{definition}


\begin{theorem}\label{bspde solution}%
Suppose that $\bf{(A1)}$ holds
and the terminal condition $q_T\in L^2_{\mf_T}(\O;L^2(D))$. Then Eq. (\ref{111}) admits a unique mild solution $(q,r)\in C_{\mathbb{F}}([0,T];L^2(\O;L^2(D)))$ $\times
L^2_{\mathbb{F}}(\O\times(0,T);L^2(D))$. In particular, for any $t\in [0,T]$,
\begin{equation*}
\begin{aligned}
\sup_{t\in[0,T]}\me\|q(t)\|_{L^2(D)}^2+\me\int_0^T\|r(s)\|^2_{L^2(D)}ds
\leq & C\Big\{\me\int_0^T\|f(s,\cdot,0,0)\|^2_{L^2(D)}ds+\me\|q_T\|^2_{L^2(D)}\Big\},\\
\end{aligned}
\end{equation*}
where $C$ depends only on $A,\, T$ and $L$.
Furthermore, if $q_T\in L^2_{\mf_T}(\O;H_0^1(D))$, then
$$(q,r)\in \big(L^2_{\mathbb{F}}(\O\times(0,T);H^2(D)\cap H_0^1(D))\cap L^2_{\mathbb{F}}(\O;C([0,T];L^2(D))\big)\times L^2_{\mathbb{F}}(\O\times(0,T);H_0^1(D)),$$
and the following estimate holds
\begin{equation*}
\begin{aligned}
&\me\big(\sup_{t\in[0,T]}\|q(t)\|_{H_0^1(D)}^2\big)+\me\int_0^T\|q(s)\|^2_{H^2(D)}+\|r(s)\|^2_{H_0^1(D)}ds\\
\leq & C\Big\{\me\int_0^T\|f(s,\cdot,0,0)\|^2_{L^2(D)}ds+\me\|q_T\|^2_{H_0^1(D)}\Big\}.\\
\end{aligned}
\end{equation*}
\end{theorem}


The next lemma provides a standard but useful estimate on the solution to stochastic differential equations (SDEs, for short). We list it for ready references.
\begin{lemma}\label{luck}
Suppose that $x(\cdot)$ solve the following SDE:
\begin{equation}\label{wang31}
\left\{
\begin{split}
dx(t)&=\big(A(t)x(t)+f(t)\big)dt+\big(B(t)x(t)+g(t)\big)dW(t),\q   t\in [0,T), \\
 \displaystyle
 x(0)&=x_0,
\end{split}
\right.
\end{equation}
where $A,\,B$ are bounded, $f(\cdot)\in L^p_{\mathbb{F}}(\O;L^1((0,T);\R^n))$ and $g(\cdot)\in L^p_{\mathbb{F}}(\O;L^2((0,T);\R^n))$. Then
$$\me\sup_{0\leq t\leq T}|x(t)|^p\leq C\bigg[\me|x_0|^p+\me\Big(\int_0^T|f(t)|dt\Big)^p+\me\Big(\int_0^T|g(t)|^2dt\Big)^{\frac{p}{2}}\bigg],$$
where $C$ is a constant depending only on $A,\,B.$
\end{lemma}

\section{Approximating BSPDE by BSDEs}

Let $\{(-\l_i,\phi_i)\}_{i=1}^{\infty}$ be the sequence of eigenvalues and
eigenfunctions of $A$, where $\{\phi_i\}_{i=1}^{\infty}$
constitutes an orthonormal basis of $L^2(D)$. 
$\{\phi_i\}_{i=1}^{\infty}$ is also an orthogonal basis of
$H_0^1(D)$. Take the subspace $S_n=\mbox{span}\{\phi_1,\cdots,\phi_n\},\,n=1,2,\cdots$.
Denote by $P_n$ the projection of $L^2(D)$ onto $S_n$ and define $A_n$ by
$A_n=P_nA|_{S_n}$.

The semidiscrete problem corresponding to Eq. (\ref{111}) is to find
a process pair $(q_n,r_n)\in  L^2_{\mathbb{F}}(\O;C([0,T];S_n))\times
L^2_{\mathbb{F}}(\O\times(0,T);S_n)$ solving the following BSDE:
\begin{equation}\label{eq:aaa}
\left\{
\begin{split}
dq_n(t)&=\big(-A_nq_n(t)+P_nf(t,\cdot,q_n(t),r_n(t))\big)dt+r_n(t)dW(t), \q t\in [0,T], \\
 \displaystyle
 q_n(T)&=P_nq_T.
\end{split}
\right.
\end{equation}

In the next theorem,  we prove that $(q_n,\,r_n)$ is convergent to $(q,\,r)$ and obtain the rate of convergence with respect to the space variable.


\begin{theorem}\label{convergence speed}%
Suppose that $\bf{(A1)}$ holds and $q_T\in L^2_{\mf_T}(\O;H_0^1(D))$.
Let $(q,r),~(q_n,r_n)$ be solutions to Eq. (\ref{111}) and (\ref{eq:aaa}),
respectively. Then the following estimate holds
\begin{equation}
\begin{aligned}
  &\me\sup_{t\in[0,T]}\|q(t)-q_n(t)\|^2_{L^2(D)}+\me\int_0^T\|q(t)-q_n(t)\|^2_{H_0^1(D)}+\|r(t)-r_n(t)\|^2_{L^2(D)}dt\\
\leq& \frac{C}{\l_{n+1}}\[\|q(T)\|^2_{L^2_{\mf_T}(\O;H_0^1(D))}
                        +\|f(\cdot,0,0)\|^2_{L^2_{\mathbb{F}}(\O\times(0,T);H_0^1(D))}\],
\end{aligned}
\end{equation}
where $C$ is a constant depending only on $T$, $L$ and $D$.

\end{theorem}

\begin{proof} We divide the proof into two steps.

\textbf{Step 1.} Applying It\^{o}'s formula to ${\lan
q-q_n,q-q_n\ran}_{L^2(D)}$,  for any $t\in[0,T]$, we have
\begin{equation}\label{bspde 10}
\begin{aligned}
 & \|q(T)-q_n(T)\|_{L^2(D)}^2 -\|q(t)-q_n(t)\|_{L^2(D)}^2\\
 =&\displaystyle
  -2\int_t^T{\lan Aq-A_nq_n,q-q_n\ran}_{L^2(D)}ds+2\int_t^T{\lan f(q,r)-P_nf(q_n,r_n),q-q_n\ran}_{L^2(D)}ds\\[3mm]
   &+2\int_t^T{\lan q-q_n,r-r_n\ran}_{L^2(D)}dW(s)+\int_t^T\|r-r_n\|^2_{L^2(D)}ds\\
\end{aligned}
\end{equation}
\begin{equation*}
\begin{aligned}
=& -2\int_t^T{\lan A(q-q_n),q-q_n\ran}_{L^2(D)} ds+2\int_t^T{\lan f(q,r)-P_nf(q_n,r_n),q-q_n\ran}_{L^2(D)}ds\\
   &+2\int_t^T{\lan q-q_n,r-r_n\ran}_{L^2(D)}dW(s)+\int_t^T\|r-r_n\|^2_{L^2(D)}ds.\\
\end{aligned}
\end{equation*}
Taking expectation in \eqref{bspde 10}, using assumption \textbf{(A1)} we obtain that
\begin{equation}\label{bspde 14}
\begin{aligned}
 & \me\|q(t)-q_n(t)\|_{L^2(D)}^2 +\me\int_t^T\|r-r_n\|^2_{L^2(D)}ds-2\me\int_t^T{\lan A(q-q_n),q-q_n\ran}_{L^2(D)} ds\\[2mm]
 =&\displaystyle
  \me\|q(T)-q_n(T)\|_{L^2(D)}^2-2\me\int_t^T{\lan f(q,r)-P_nf(q_n,r_n),q-q_n\ran}_{L^2(D)}ds\\
 =&\me\|q(T)-q_n(T)\|_{L^2(D)}^2-2\me\int_t^T{\lan f(s,q,r)-P_nf(s,q,r),q-q_n\ran}_{L^2(D)}ds\\
&-2\me\int_t^T{\lan P_nf(s,q,r)-P_nf(s,q_n,r),q-q_n\ran}_{L^2(D)}ds\\
&-2\me\int_t^T{\lan P_nf(s,q_n,r)-P_nf(s,q_n,r_n),q-q_n\ran}_{L^2(D)}ds\\
\leq&\me\|q(T)-q_n(T)\|_{L^2(D)}^2+2L\me\int_t^T\|(I-P_n)q\|_{L^2(D)}\|q-q_n\|_{L^2(D)}ds\\
&+2L\me\int_t^T\|(I-P_n)r\|_{L^2(D)}\|q-q_n\|_{L^2(D)}ds\\
&+2\me\int_t^T\|(I-P_n)f(s,0,0)\|_{L^2(D)}\|q-q_n\|_{L^2(D)}ds\\
&+\me\int_t^T(2L+1)\|q-q_n\|^2_{L^2(D)}+\|r-r_n\|^2_{L^2(D)}ds\\
\leq&TL\max_{t\in[0,T]}\me\|(I-P_n)q(t)\|^2_{L^2(D)}
+L\me\int_t^T\|(I-P_n)r(s)\|^2_{L^2(D)}ds\\
&+\me\int_t^T\|(I-P_n)f(s,0,0)\|^2_{L^2(D)}ds
+(4L+2)\me\int_t^T\|q-q_n\|^2_{L^2(D)}ds\\
&\me\int_t^T\|r-r_n\|^2_{L^2(D)}ds.\\
\end{aligned}
\end{equation}
Since for any $t\in[0,T]$,
$$-2\me\int_t^T{\lan A(q-q_n),q-q_n\ran}_{L^2(D)} ds\geq 0,$$
by Gronwall's inequality, one can easily check that
\begin{equation}\label{bspde 15}
\begin{aligned}
     & \sup_{t\in[0,T]}\me\|q(t)-q_n(t)\|^2_{L^2(D)}\\
\displaystyle
\leq & \displaystyle
e^{(2L+1)T}\[\me\|q(T)-q_n(T)\|^2_{L^2(D)}
 +TL\max_{t\in[0,T]}\me\|(I-P_n)q(t)\|^2_{L^2(D)}\\
&\,\,+L\me\int_0^T\|(I-P_n)r(t)\|^2_{L^2(D)}dt
+\me\int_0^T\|(I-P_n)f(t,0,0)\|^2_{L^2(D)}dt\].\\
\end{aligned}
\end{equation}

Set $q(\cdot)=\sum_{j=1}^{\infty}\mu_{j}(\cdot)\phi_j$. Then
$q_n(\cdot)=P_nq(\cdot)=\sum_{j=1}^{n}\mu_{j}(\cdot)\phi_j$. Hence
\begin{equation}\label{bspde 12}
\begin{aligned}
     & \me\|q(\cdot)-q_n(\cdot)\|^2_{L^2(D)}=\me\|(I-P_n)q(\cdot)\|^2_{L^2(D)}
 =  {\sum_{i=n+1}^{\infty}\me|\mu_i(\cdot)|^2}\leq
       {\frac{1}{\l_{n+1}}\sum_{i=n+1}^{\infty}\l_i\me|\mu_i(\cdot)|^2}\\
\leq &  {\frac{1}{\l_{n+1}}\sum_{i=1}^{\infty}\l_i\me|\mu_i(\cdot)|^2}
\leq   {\frac{1}{\l_{n+1}}\me\|q(\cdot)\|^2_{H_0^1(D)}}
=   {\frac{1}{\l_{n+1}}}\|q(\cdot)\|^2_{L^2_{\mf_\cdot}(\O;H_0^1(D))}.%
\end{aligned}
\end{equation}
Similarly, 
\begin{equation}\label{bspde 13}
\begin{aligned}
\me\int_0^T\|(I-P_n)r(t)\|^2_{L^2(D)}dt
&\leq {\frac{1}{\l_{n+1}}}\me\int_0^T\|r(t)\|^2_{H_0^1(D)}dt,\\
\me\int_0^T\|(I-P_n)f(t,0,0)\|^2_{L^2(D)}dt
&\leq {\frac{1}{\l_{n+1}}}\me\int_0^T\|f(t,0,0)\|^2_{H_0^1(D)}dt.
\end{aligned}
\end{equation}
From \eqref{bspde 14}--\eqref{bspde 13} and Theorem \ref{bspde solution}, we have
\begin{equation}\label{bspde 16}
\begin{aligned}
&\sup_{t\in[0,T]}\me\|q(t)-q_n(t)\|^2_{L^2(D)}\\
\leq&
\frac{e^{3(2L+1)T}}{\l_{n+1}}\[\|q(T)\|^2_{L^2_{\mf_T}(\O;H_0^1(D))}+TL\max_{t\in[0,T]}\|q(t)\|^2_{L^2_{\mf_t}(\O;H_0^1(D))}\\
&\q\q+L\|r\|^2_{L^2_{\mathbb{F}}(\O\times(0,T);H_0^1(D))}
                        +\|f(\cdot,0,0)\|^2_{L^2_{\mathbb{F}}(\O\times(0,T);H_0^1(D))}\]\\
\leq & \frac{C}{\l_{n+1}}\[\|q(T)\|^2_{L^2_{\mf_T}(\O;H_0^1(D))}
       +\|f(\cdot,0,0)\|^2_{L^2_{\mathbb{F}}(\O\times(0,T);H_0^1(D))}\],\\
\end{aligned}
\end{equation}
where $C$ is a constant independently on $\l_n$. By the same argument, we obtain that
\begin{equation}\label{bspde 17}
\begin{aligned}
                        &\me\int_0^T\|r(t)-r_n(t)\|^2_{L^2(D)}dt-2\me\int_0^T{\lan A(q-q_n)(t),(q-q_n)(t)\ran}_{L^2(D)} dt\\
\leq&\frac{C}{\l_{n+1}}\[\|q(T)\|^2_{L^2_{\mf_T}(\O;H_0^1(D))}
    +\|f(\cdot,0,0)\|^2_{L^2_{\mathbb{F}}(\O\times(0,T);H_0^1(D))}\].\\
 \end{aligned}
 \end{equation}

\textbf{Step 2.}
Using \eqref{bspde 10} once more and from similar proceeding as that in Step 1, by the Burkholder-Davis-Gundy inequality, we have
\begin{equation}\label{bspde 11}
\begin{aligned}
 &\me \sup_{t\in[0,T]}\|q(t)-q_n(t)\|_{L^2(D)}^2\\
 =&
  \me\sup_{t\in[0,T]}\Bigg\{\|q(T)-q_n(T)\|_{L^2(D)}^2+2\int_t^T{\lan A(q-q_n),q-q_n\ran}_{L^2(D)} ds\\
   & -2\int_t^T{\lan f(q,r)-P_nf(q_n,r_n),q-q_n\ran}_{L^2(D)}ds\\
   &-2\int_t^T{\lan q-q_n,r-r_n\ran}_{L^2(D)}dW(s)
    -\int_t^T\|r-r_n\|^2_{L^2(D)}ds\Bigg\}\\
\end{aligned}
\end{equation}
\begin{equation*}
\begin{aligned}
   \leq&\me\|q(T)-q_n(T)\|_{L^2(D)}^2+TL\max_{t\in[0,T]}\me\|(I-P_n)q(t)\|^2_{L^2(D)}\\
   &+L\me\int_t^T\|(I-P_n)r(s)\|^2_{L^2(D)}ds
      +\me\int_t^T\|(I-P_n)f(s,0,0)\|_{L^2(D)}^2ds\\
&+\me\int_t^T(4L+2)\|q-q_n\|^2_{L^2(D)}ds+2\me\sup_{t\in[0,T]}\int_t^T{\lan q-q_n,r-r_n\ran}_{L^2(D)}dW(s)\\
\leq&\me\|q(T)-q_n(T)\|_{L^2(D)}^2
    +TL\max_{t\in[0,T]}\me\|(I-P_n)q(t)\|^2_{L^2(D)}\\
&+L\me\int_t^T\|(I-P_n)r(s)\|^2_{L^2(D)}ds
   +\me\int_t^T\|(I-P_n)f(s,0,0)\|^2_{L^2(D)}ds\\
&+\me\int_t^T(4L+2)\|q-q_n\|^2_{L^2(D)}
+6\me\[\int_0^T{\lan q-q_n,r-r_n\ran}_{L^2(D)}^2dt\]^{1/2}\\
\leq&\me\|q(T)-q_n(T)\|_{L^2(D)}^2+TL\max_{t\in[0,T]}\me\|(I-P_n)q(t)\|^2_{L^2(D)}\\
&+L\me\int_t^T\|(I-P_n)r(s)\|^2_{L^2(D)}ds+\me\int_t^T\|(I-P_n)f(s,0,0)\|^2_{L^2(D)}ds\\
&+(4L+2)\me\int_t^T\|q-q_n\|^2_{L^2(D)}ds+18\me\int_t^T\|r-r_n\|^2_{L^2(D)}ds\\
&+\frac{1}{2}\me\sup_{t\in[0,T]}\|q(t)-q_n(t)\|_{L^2(D)}^2.\\
\end{aligned}
\end{equation*}
From \eqref{bspde 12}--\eqref{bspde 11} and Theorem \ref{bspde solution}, we obtain
\begin{equation*}\label{bspde 11s}
\begin{aligned}
 \me \sup_{t\in[0,T]}\|q(t)-q_n(t)\|_{L^2(D)}^2
 \leq\frac{C}{\l_{n+1}}\[\|q(T)\|^2_{L^2_{\mf_T}(\O;H_0^1(D))}+\|f(\cdot,0,0)\|^2_{L^2_{\mathbb{F}}(\O\times(0,T);H_0^1(D))}\],\\
 \end{aligned}
\end{equation*}
where $C$ depends only on $T$, $L$ and $D$. This completes the proof.
\end{proof}
\medskip




\begin{remark}
Under the assumptions of Theorem \ref{convergence speed}, applying the results in \cite[Section 4.1]{Tessitore}, we can prove a better convergence result than that in Theorem \ref{convergence speed}, i.e.,
$$(q_n,\,r_n)\longrightarrow (q,\,r)\q \mbox{in}\q L^2_{\mathbb{F}}(\O\times(0,T);H^2(D))\times L^2_{\mathbb{F}}(\O\times(0,T);H^1(D)).$$
However, the convergence speed (given in Theorem \ref{convergence speed}) cannot be improved.

\end{remark}

\section{Approximation for BSDE}

In this section, we apply the backward Euler scheme to solve Eq. \eqref{eq:aaa}. We borrow
some idea from \cite{Hu-Nualart-Song,Zhang04}.

Since $S_n=\mbox{span}\{\phi_1,\cdots,\phi_n\}$, and the solution $(q_n,r_n)$ to Eq. \eqref{eq:aaa} is in $L^2_{\mathbb{F}}(\O\times(0,T);S_n)\times L^2_{\mathbb{F}}(\O\times(0,T);S_n)$, we may take $(q_n,r_n)$ to be the following form
\begin{equation}\label{expansion 1}
\begin{aligned}
q_n(\cdot)=\sum_{j=1}^n\a_{n,j}(\cdot)\phi_j,\,~
r_n(\cdot)=\sum_{j=1}^n\b_{n,j}(\cdot)\phi_j,
\end{aligned}
\end{equation}
where $\alpha_{n,j}\in L^2_{\mathbb{F}}(\O;C([0,T]; \R))$ and $\beta_{n,j}\in L^2_{\mathbb{F}}(\O\times(0,T); \R).$
Set
\begin{equation*}
\begin{aligned}
\alpha_n(\cdot)=\begin{pmatrix} \alpha_{n,1}(\cdot)\\\alpha_{n,2}(\cdot)\\\vdots
\\\alpha_{n,n}(\cdot)
\end{pmatrix},
\q \beta_n(\cdot)=\begin{pmatrix} \beta_{n,1}(\cdot)\\\beta_{n,2}(\cdot)\\\vdots
\\\beta_{n,n}(\cdot)
\end{pmatrix},
\q \L=\begin{pmatrix}
\l_1& 0  &\cdots &0\\
0   & \l_2&\cdots&0\\
\vdots&\vdots&\ddots&\vdots\\
0&  0     &\cdots &\l_n\\
\end{pmatrix},\\
\end{aligned}
\end{equation*}
and $\ds f_n(\cdot,\alpha_n(\cdot),\beta_n(\cdot))=\begin{pmatrix}
\lan f(\cdot,q_n(\cdot),r_n(\cdot)),\phi_1 \ran_{L^2(D)}\\\lan f(\cdot,q_n(\cdot),r_n(\cdot)),\phi_2 \ran_{L^2(D)}\\\vdots\\\lan f(\cdot,q_n(\cdot),r_n(\cdot)),\phi_n \ran_{L^2(D)}\\
\end{pmatrix}.$\\
Then $(\alpha_n(\cdot),\beta_n(\cdot))$ solves the following BSDE:
\begin{equation}\label{bsden}
\left\{
\begin{aligned}
d\alpha_n(t)&=\big(\L \alpha_n(t)+f_n(t,\alpha_n(t),\beta_n(t))\big)dt+\beta_n(t)dW(t), \q t\in [0,T], \\
 \displaystyle
 \alpha_n(T)&=\alpha_{nT}:=\big({\lan q_T,\phi_1 \ran}_{L^2(D)},\,{\lan q_T,\phi_2 \ran}_{L^2(D)},\,\cdots,\,{\lan q_T,\phi_n \ran}_{L^2(D)}\big)^*.
\end{aligned}
\right.
\end{equation}

Before presenting the main theorem, 
we study the regularity of $\alpha_{n}(\cdot)$ and $\beta_n(\cdot)$. In the sequel, for simplicity, we write $ \partial_y f_n(\cdot,\a_n(\cdot),\b_n(\cdot)), \, \partial_z f_n(\cdot,\a_n(\cdot),\b_n(\cdot)) $ by $\partial_{y} f_n(\cdot),\,\partial_{z} f_n(\cdot)$ respectively. Similarly, we may use the notations $\partial_{yy}f_n(\cdot)$, $\partial_{yz}f_n(\cdot)$ and $\partial_{zz}f_n(\cdot)$.

\begin{lemma}\label{alpha}
Let $\bf{(A1)}$ and $\bf{(A2)}$ hold.
Then, for any $t,s \in [0,T]$,
$$\me|\alpha_{n}(t)-\alpha_{n}(s)|^2\leq C(1+\l_n^2|t-s|)|t-s|,$$
where $C$ is a constant depending only on $q,\,T$, $L$ and $M$.

\end{lemma}

\begin{proof}
Under the assumptions on $q_T$ and $f$, by \cite[Theorem 2.6]{Hu-Nualart-Song},
$(\alpha_{n},\beta_{n})$ is in $\mathbb{L}^{1,2}_a(\R^n)$. Besides, the Malliavin derivative
$(D_{\theta}\alpha_{n},D_{\theta}\beta_{n})$ of the solution pair solves the following BSDE:
\begin{equation}\label{M1}
\left\{
\begin{aligned}
D_{\theta}\alpha_{n}(T)-D_{\theta}\alpha_{n}(t)&=\int_t^T\big(\L D_{\theta}\alpha_{n}(s)+\partial_yf_n(s)D_{\theta}\alpha_n(s)
+\partial_zf_n(s)D_{\theta}\beta_n(s)\big)\\
&\qq+\int_t^T D_{\theta}\beta_{n}(s)dW(s), \q 0\leq\theta\leq t\leq T;\\
D_{\theta}\alpha_{n}(t)=D_{\theta}\beta_{n}(t)&=0, \q 0\leq t< \theta\leq T.
\end{aligned}
\right.
\end{equation}
Moreover, $\beta_{n}(\cdot)$ can be represented by $D_{\cdot}\alpha_{n}(\cdot)$. Furthermore, 
\begin{equation}\label{xiao}
\begin{aligned}
\sup_{0\leq\theta\leq T}\me \sup_{\theta\leq t\leq T}|D_{\theta}\alpha_n(t)|^q+\me\bigg(\int_{\theta}^T|D_{\theta}\beta_n(t)|^2dt\bigg)^{q/2}
\leq C\me|D_{\theta}\alpha_n(T)|^q,
\end{aligned}
\end{equation}
where $C$ depends only on $q,\,T$ and $L$.

By above inequality,  we have
\begin{equation}\label{M2}
\me|\beta_{n}(t)|^2=\me|D_{t}\alpha_{n}(t)|^2\leq \big(\me|D_{t}\alpha_{n}(t)|^q\big)^{2/q}\leq C \big(\me|D_{t}\alpha_{n}(T)|^q\big)^{2/q},\q \ae t\in [0,T].
\end{equation}
%
By \eqref{bsden}, we also know that
\begin{equation}\label{M3}
\me|\alpha_{n}(t)|^2+\me\int_t^T|\beta_{n}(s)|^2ds\leq C \Big(\me|\alpha_{n}(T)|^2+\me\int_t^T|f_n(s,0,0)|^2ds\Big).
\end{equation}
Here $C$ is a constant depending only on $L$ and $T$.
By virtue of \eqref{M2} and \eqref{M3}, for $0\leq s\leq t\leq T$, we can easily have
\begin{equation*}
\begin{aligned}
&\me|\alpha_{n}(t)-\alpha_{n}(s)|^2=\me\bigg|\int_s^t(\L\alpha_{n}(u)+f_n(u,\alpha_n(u),\beta_n(u))du+\int_s^t\beta_{n}(u)dW(u))\bigg|^2\\
\leq&\Big(4\me\int_s^t\l_n^2|\alpha_{n}(u)|^2du+16L\int_s^t(|\alpha_n(u)|^2+|\beta_n(u)|^2)du+8\me\int_s^t f_{n}(u,0,0)|^2du\Big)(t-s)\\
&\q+2\me\int_s^t\beta_{n}^2(u)du\\
\leq&\big(4\l_n^2(t-s)+16L(t-s+1)+8\big)C\Big(\me|\alpha_{n}(T)|^2+\me\int_t^T|f_n(s,0,0)|^2ds\Big)(t-s)\\
&+2C\me|D_{t}\alpha_{n}(T)|^2(t-s)\\
\leq & C(1+\l_n^2(t-s))(t-s),
\end{aligned}
\end{equation*}
where $C$ depends only on $q,\,L$, $T$ and $M$.
This completes the proof.
\end{proof}

The following lemma is about the regularity of $\beta_{n}(\cdot)$. 
\begin{lemma}\label{beta}
Let $\bf{(A1)}$ and $\bf{(A2)}$ hold.
Then, for any $t,s \in [0,T]$,
\begin{equation}\label{beta10}
\begin{aligned}
\me|\beta_{n}(t)-\beta_{n}(s)|^2\leq Ce^{\l_n |t-s|}|t-s|(\l_n^2|t-s|+1),
\end{aligned}
\end{equation}
where $C$ is a constant depending only on $q,\, T$, $L$ and $M$.

\end{lemma}

We need the following lemma to prove the above result. 
\begin{lemma}\label{wang65}
Let $\bf{(A1)}$ and $\bf{(A2)}$ hold, and $\{\Psi\}_{0\leq t\leq T}$ and $\{\Phi\}_{0\leq t\leq T}$ which solve the following SDEs
\begin{equation}\label{wang1}
\left\{
\begin{split}
d\Psi(t)&=-\Psi(t)\big(\L+\partial_yf_n(t)\big)dt-\Psi(t)\partial_zf_n(t)dW(t),\q   t\in [0,T), \\
 \displaystyle
 \Psi(0)&=I_n
\end{split}
\right.
\end{equation}
and 
\begin{equation}\label{wang2}
\left\{
\begin{split}
d\Phi(t)&=\big(\L+\partial_yf_n(t)+\big(\partial_yf_n(t)\big)^2\big)\Phi(t)dt+\partial_zf_n(t)\Phi(t)dW(t),\q   t\in [0,T), \\
 \displaystyle
 \Phi(0)&=I_n,
\end{split}
\right.
\end{equation}
respectively.  Then, for any $p\geq 2$, 
\begin{equation}\label{wang66}
\begin{aligned}
\me\sup_{0\leq t\leq T}|\Psi(t)|^p\leq C,\q\me\sup_{0\leq t\leq T}|\Phi(t)|^p\leq Ce^{\l_n T};\\
\end{aligned}
\end{equation}
\begin{equation}\label{wang41}
\begin{aligned}
\me\big|\sup_{s\leq t\leq T}\Phi(s)\Psi(t)\big|^p\leq C,\q \me|\Phi(t)\Psi(s)|^p\leq Ce^{\l_n |t-s|},\\
\end{aligned}
\end{equation}
\begin{equation}\label{wang45}
\begin{aligned}
\me|(\Phi(t)-\Phi(s))\Psi(T)|^p\leq C e^{\l_n |t-s|}|t-s|^{\frac{p}{2}}(\l_n^p|t-s|^{\frac{p}{2}}+1),
\end{aligned}
\end{equation}
where $C$ depends only on $p,\,L$ and $T$.
\end{lemma}

\begin{proof}
First of all,
for any $x_0\in \R^n$, set $x(\cdot)=\Psi^*(\cdot)x_0$. Then $x(\cdot)$ solves the following SDE:
\begin{equation*}\label{wang4}
\left\{
\begin{split}
dx(t)&=\big(-\L-\partial_yf_n(t)\big)^*x(t)dt-\partial_zf_n^*(t)x(t)dW(t),\q   t\in [0,T), \\
 \displaystyle
 x(0)&=x_0.
\end{split}
\right.
\end{equation*}
By It\^o's formula, one can have
\begin{equation*}\label{bai11}
\begin{aligned}
&|x(t)|^p-|x(0)|^p\\
=&\int_0^t p|x|^{p-2}\lan x, (-\L-\partial_yf_n^*x\ran ds+\int_0^t p|x|^{p-2}\lan x,-\partial_zf_n^*x  \ran dW(s)\\
 &+\frac{1}{2}\int_0^t x^*\partial_zf_n\partial_{xx}|x|^p\partial_zf_n^*x ds\\
 \leq&\int_0^t p|x|^{p-2}\lan x, -\partial_yf_n(t)^*x\ran ds+\int_0^t p|x|^{p-2}\lan x,-\partial_zf_n^*x  \ran dW(s)\\
 &+\frac{1}{2}\int_0^t x^*\partial_zf_n\partial_{xx}|x|^p\partial_zf_n^*x ds\\
 \leq&C\int_0^t|x|^pds-p\int_0^t|x|^{p-2}\lan x,-\partial_zf_n^*x  \ran dW(s).\\
\end{aligned}
\end{equation*}
Then, 
$$\me\sup_{0\leq t\leq T}|x(t)|^p\leq C|x_0|^p;$$
therefore
$$\me\sup_{0\leq t\leq T}|\Psi(t)|^p=\sup_{x_0\in\R^n}\frac{\me\sup_{0\leq t\leq T}|x(t)|^p}{|x_0|^p}\leq C$$
here $C$ depends only on $p,\,L$ and $T$. Similarly, one can prove $\ds \me\sup_{0\leq t\leq T}|\Phi(t)|^p\leq Ce^{\l_n T}$, and then \eqref{wang66} is proved.\\

Next, we only prove the second inequality of \eqref{wang41}. The first one can be proved with the similar procedure.
For any $x_0\in \R^n$, set $x_s(t)=\Phi(t)\Psi(s)x_0$. Then $x_s(t)$ solves the following stochastic differential equation:
\begin{equation*}\label{wang42}
\left\{
\begin{split}
dx_s(t)&=\big(\L+\partial_yf_n+(\partial_yf_n)^2\big)x_s(t)dt+\partial_zf_nx_s(t)dW(t),\q   t\in [s,T), \\
 \displaystyle
 x_s(s)&=x_0.
\end{split}
\right.
\end{equation*}
By It\^o's formula, one can have
\begin{equation*}\label{bai41}
\begin{aligned}
&|x_s(t)|^p-|x_s(s)|^p\\
=&\int_s^t p|x_s(\tau)|^{p-2}\lan x_s(\tau), (\L\partial_yf_n(\tau)+(\partial_yf_n(\tau))^2)x_s(\tau)\ran d\tau\\
  &\q+\int_s^t p|x_s(\tau)|^{p-2}\lan x_s(\tau),\partial_zf_n x_s(\tau)  \ran dW(\tau)\\
  &\q
 +\frac{1}{2}\int_s^t x_s(\tau)^*\partial_zf_n^*\partial_{xx}|x_s(\tau)|^p\partial_zf_n x_s(\tau) d\tau\\
 \leq&C(\l_n+1)\int_s^t|x_s(\tau)|^pd\tau+C\int_s^t|x_s(\tau)|^{p-2}\lan x_s(\tau),\partial_zf_n x_s(\tau)  \ran dW(\tau).\\
\end{aligned}
\end{equation*}
Then, by Gronwall's inequality, 
$$\me|\Phi(t)\Psi(s)x_0|^p=\me|x_s(t)|^p\leq Ce^{\l_n(t-s)}|x_0|^p,$$
where $C$ depends only on $p,\,L$.\\

Finally, by Eq. \eqref{wang2}, one has
\begin{equation}\label{bai12}
\begin{aligned}
&\me|(\Phi(t)-\Phi(s))\Psi(T)|^p\\
=&\me\Big|\int_s^t(\L+\partial_yf_n(\tau)+(\partial_yf_n(\tau))^2)\Phi(\tau) d\tau\Psi(T)+\int_s^t\partial_zf_n(\tau)\Phi(\tau) dW(\tau)\Psi(T)\Big|^p\\
\leq&C\l_n^p\me\Big(\int_s^t\big|\Phi(\tau)\Psi(T)\big|d\tau\Big)^p+C\me\Big(\int_s^t\big|\Phi(\tau)\Psi(T)\big|d\tau\Big)^p\\
&\q+C\me\Big|\int_s^t\partial_zf_n(\tau)\Phi(\tau) dW(\tau)\Psi(T)\Big|^p\\
\leq&C(\l_n^p+1)\me\Big(\int_s^t\big|\Phi(\tau)\Psi(T)\big|d\tau\Big)^p+C\me\Big|\int_s^t\partial_zf_n(\tau)\Phi(\tau) dW(\tau)\Psi(T)\Big|^p\\
:=&C(\l_n^p+1)J_1+CJ_2.
\end{aligned}
\end{equation}

For $J_1$, by \eqref{wang41}, we have
\begin{equation}\label{bai43}
\begin{aligned}
J_1\leq& \me\int_s^t|\Phi(\tau)\Psi(T)|^pd\tau\Big(\int_s^t1d\tau\Big)^{\frac{p}{q}}
=\int_s^t\me|\Phi(\tau)\Psi(T)|^pd\tau(t-s)^{\frac{p}{q}}\\
\leq& C(t-s)^{1+\frac{p}{q}}
=C(t-s)^p,
\end{aligned}
\end{equation}
where $C$ depends only on $p,\,L$ and $T$.

For $J_2$, by \eqref{wang41}, H\"older's inequality and Burkholder-Davis-Gundy inequality, we also can obtain,
\begin{equation}\label{bai44}
\begin{aligned}
J_2=&\me\Big|\int_s^t\partial_zf_n(\tau)\Phi(\tau) dW(\tau)\Psi(T)\Big|^p=\me\Big|\int_s^t\partial_zf_n(\tau)\Phi(\tau)\Psi(s) dW(\tau)\Phi(s)\Psi(T)\Big|^p\\
\leq & \Big(\me\Big|\int_s^t\partial_zf_n(\tau)\Phi(\tau)\Psi(s)dW(\tau)\Big|^{2p}\Big)^{1/2}\big(\me|\Phi(s)\Psi(T)|^{2p}\big)^{1/2}\\
\leq &C\Big(\me\Big(\int_s^t|\Phi(\tau)\Psi(s)|^2dW(\tau)\Big)^{p}\Big)^{1/2}\big(\me|\Phi(s)\Psi(T)|^{2p}\big)^{1/2}\\
\leq& C\Big\{\me\Big[\Big(\int_s^t1d\tau\Big)^{\frac{1}{q}}\Big(\int_s^t|\Phi(\tau)\Psi(s)|^{2p}d\tau\Big)^{\frac{1}{p}}\big]^p\Big\}^{1/2}\big(\me|\Phi(s)\Psi(T)|^{2p}\big)^{1/2}\\
\leq& C\Big\{(t-s)^{\frac{p}{q}}\int_s^t\me|\Phi(\tau)\Psi(s)|^{2p}d\tau\Big\}^{1/2}\big(\me|\Phi(s)\Psi(T)|^{2p}\big)^{1/2}\\
\leq & C\Big\{(t-s)^{\frac{p}{q}}\Big[e^{\l_n(t-s)}(t-s)\Big]\Big\}^{1/2}\\
\leq & C e^{\l_n(t-s)}(t-s)^{\frac{p}{2}},
\end{aligned}
\end{equation}
where $C$ depends only on $p,\,L$.
Combining \eqref{bai12}-\eqref{bai44}, we have \eqref{wang45}.
\end{proof}

We are now in a position to prove Lemma
\ref{beta}.

\begin{proof}[\bf {Proof of Lemma~\ref{beta}}]

Let $0\leq s\leq t\leq T$.
Since $\beta_n(\cdot)$ has the representation $\beta_n(t)=D_{t}\alpha_n(t),\, \as,\, \ae t\in [0,T],$ we have
\begin{equation}\label{beta1}
\begin{aligned}
\beta_n(t)-\beta_n(s)=(D_t\alpha_n(t)-D_s\alpha_n(t))+(D_s\alpha_n(t)-D_s\alpha_n(s)).
\end{aligned}
\end{equation}
By Eq. \eqref{M1}, for any $\theta,\,\theta'\in[0,t] $, by It\^o's formula, we can obtain that
\begin{equation*}\label{beta2}
\begin{aligned}
&\me|D_{\theta}\a_n(t)-D_{\theta'}\a_n(t)|^2+\me\int_t^T|D_{\theta}\b_n(s)-D_{\theta'}\b_n(s)|^2ds\\
=&\me|D_{\theta}\a_n(T)-D_{\theta'}\a_n(T)|^2-2\int_t^T\big(\lan D_{\theta}\a_n(s)-D_{\theta'}\a_n(s), \L (D_{\theta}\a_n(s)-D_{\theta'}\a_n(s))\ran\\
 &+\lan D_{\theta}\a_n(s)-D_{\theta'}\a_n(s), \partial_yf_n(s)(D_{\theta}\a_n(s)-D_{\theta'}\a_n(s))
+\partial_zf_n(s)(D_{\theta}\b_n(s)-D_{\theta'}\b_n(s))\big)ds\\
\leq &\me|D_{\theta}\a_n(T)-D_{\theta'}\a_n(T)|^2+C\int_t^T|D_{\theta}\a_n(s)-D_{\theta'}\a_n(s)|^2ds\\
&+\frac{1}{2}\int_t^T|D_{\theta}\b_n(s)-D_{\theta'}\b_n(s)|^2ds,\\
\end{aligned}
\end{equation*}
where $C$ is independent of $n$. Here, we apply the fact that $\L$ is positive and $\partial_y f_n(\cdot),\,\partial_z f_n(\cdot)$ are bounded. Therefore, setting $\theta=t,\,\theta'=s$, by Gronwall's inequality and $\bf{(A2)}$ , we deduce that
\begin{equation}\label{beta3}
\begin{aligned}
\me|D_{t}\a_n(t)-D_{s}\a_n(t)|^2
\leq C\me|D_{t}\a_n(T)-D_{s}\a_n(T)|^2
\leq C|t-s|,
\end{aligned}
\end{equation}
where $C$ is independent of $n$.\\

For  $\{\Psi(t)\}_{0\leq t\leq T}$ and $\{\Phi(t)\}_{0\leq t\leq T}$  introduced in  \eqref{wang1} and \eqref{wang2}, from \cite[Theorem 6.14]{Yong-Zhou99}, $\Psi(t)\Phi(t)=I_n$, $\ae\,t\in[0,T]$.
By virtue of $\Psi(\cdot)$ and $\Phi(\cdot)$, we can list a representation of $D_\theta\a(\cdot)$. 
By It\^o's formula,
\begin{equation*}\label{bai1}
\begin{aligned}
&d(\Psi(t)D_{\theta}\a(t))\\
=&-\Psi(t)(\L+\partial_yf_n(t))D_{\theta}\a_n(t)dt-\Psi(t)\partial_zf_n(t)D_{\theta}\a_n(t)dW(t)\\
 &+\Psi(t)(\L+\partial_yf_n(t))D_{\theta}\a_n(t)dt+\Psi(t)\partial_zf_n(t)D_{\theta}\b_n(t)dt\\
 &+\Psi(t) D_{\theta}\b_n(t)dW(t)-\Psi(t)\partial_zf_n(t)D_{\theta}\b_n(t)dt\\
=&\Psi(t)\big(D_{\theta}\b_n-\partial_zf_n(t)D_{\theta}\a_n(t)\big)dW(t).
\end{aligned}
\end{equation*}
Therefore, 
\begin{equation*}\label{bai2}
\begin{aligned}
&D_{\theta}\a(t))\\
=&\Phi(t)\Psi(T)D_{\theta}\a_n(T)-\Phi(t)\int_t^T \Psi(s)(D_{\theta}\b_n(s)-\partial_zf_n(s)D_{\theta}\a_n(s))dW(s)\\
=&\me(\Phi(t)\Psi(T)D_{\theta}\a_n(T)|\mf_t),
\end{aligned}
\end{equation*}
and
\begin{equation}\label{bai3}
\begin{aligned}
&\me|D_{\theta}\a_n(t)-D_{\theta}\a_n(s)|^2\\
=&\me\big|\me\big(\Phi(t)\Psi(T)D_{\theta}\a_n(T)|\mf_t\big)-\me(\Phi(s)\Psi(T)D_{\theta}\a_n(T)|\mf_s)\big|^2\\
\leq&2\me\big|\me\big((\Phi(t)-\Phi(s))\Psi(T)D_{\theta}\a_n(T)|\mf_t\big)\big|^2\\
&+2\me\big|\me\big(\Phi(s)\Psi(T)D_{\theta}\a_n(T)|\mf_t\big)-\me\big(\Phi(s)\Psi(T)D_{\theta}\a_n(T)|\mf_s\big)\big|^2\\
\leq&2\me\big|(\Phi(t)-\Phi(s))\Psi(T)D_{\theta}\a_n(T)\big|^2\\
&+2\me\big|\me\big(\Phi(s)\Psi(T)D_{\theta}\a_n(T)|\mf_t\big)-\me\big(\Phi(s)\Psi(T)D_{\theta}\a_n(T)|\mf_s\big)\big|^2\\
:=&2(I_1+I_2).
\end{aligned}
\end{equation}
By H\"older's inequality, \textbf{(A2)} and \eqref{wang45},  we obtain
\begin{equation}\label{bai4}
\begin{aligned}
I_1\leq& \big(\me|D_{\theta}\a_n(T)|^q\big)^{2/q}\big(\me(\Phi(t)-\Phi(s))\Psi(T)|^{\frac{2q}{q-2}}\big)^{\frac{q-2}{q}}\\
\leq& Ce^{\l_n |t-s|}|t-s|(\l_n^2|t-s|+1),
\end{aligned}
\end{equation}
where $C$ depends only on $q,\,L,\,T$ and $M$.

We claim that for any $0\leq \theta,\,s\leq T$, 
\begin{equation}\label{m22}
\begin{aligned}
\Phi(s)\Psi(T)D_{\theta}\a_n(T)\in M^{2,2},
\end{aligned}
\end{equation}
i.e., $\Phi(s)\Psi(T)D_{\theta}\a_n(T)$ has a representation:
$$\Phi(s)\Psi(T)D_{\theta}\a_n(T)
=\me(\Psi(T)D_{\theta}\a_n(T))+\int_0^Tu_{\theta,s}(t)dW(t),$$
and $\sup_{\theta,s}\me|u_{\theta,s}(t)|^{2}
\leq C$,
where $C$ depends only on $q$, $L,\,T$ and $M$.

If \eqref{m22} is true, then
\begin{equation*}\label{bai5}
\begin{aligned}
I_2=&\me\big|\me(\Phi(s)\Psi(T)D_{\theta}\a_n(T)|\mf_t)-\me(\Phi(s)\Psi(T)D_{\theta}\a_n(T)|\mf_s)\big|^2\\
=&\me\Big|\int_s^tu_{\theta,s}(\tau)dW(\tau)\Big|^2=\me\int_s^t|u_{\theta,s}(\tau)|^2d\tau\\
\leq & C|t-s|,
\end{aligned}
\end{equation*}
where $C$ depends only on $q,\,L,\,T$ and $M$.
Combining the above inequality and \eqref{beta1}-\eqref{bai4} results in \eqref{beta10}.\\

Now we prove \eqref{m22}. Indeed,
For any $0\leq \theta,\,s\leq t\leq T$,
$D_{\theta}(\Phi(s)\Psi(\cdot))$ satisfies the following SDE:
\begin{equation*}\label{wang12}
\left\{
\begin{split}
dD_{\theta}(\Phi(s)\Psi(t))&=D_{\theta}(\Phi(s)\Psi(t))\big(-\L-\partial_yf_n(t)\big)dt-D_{\theta}(\Phi(s)\Psi(t))\partial_zf_n(t)dW(t)\\
           &\qq-(\Phi(s)\Psi(t))(\partial_{yy}f_n(t)D_{\theta}\a_n(t)+\partial_{yz}f_n(t)D_{\theta}\b_n(t))dt\\
           &\qq-(\Phi(s)\Psi(t))(\partial_{yz}f_n(t)D_{\theta}\a_n(t)+\partial_{zz}f_n(t)D_{\theta}\b_n(t))dW(s),\q   \theta\leq t\leq T, \\
 \displaystyle
 D_{\theta}(\Phi(s)\Psi(\theta))&=0,\\
 D_{\theta}(\Phi(s)\Psi(t))&=0,\q 0\leq t <\theta.
\end{split}
\right.
\end{equation*}
For any $x_0\in \R^n$, set $x_{\theta,s}(\cdot)=D_{\theta}(\Psi^*(\cdot)\Phi^*(s))x_0$ and $y_s(\cdot)=\Psi^*(\cdot)\Phi^*(s)x_0$. Then $x_{\theta,s}(\cdot)$ satisfies the following SDE:
\begin{equation*}\label{wang13}
\left\{
\begin{split}
dx_{\theta,s}(t)&=\big(-\L-\partial_yf_n^*(t)\big)x_{\theta,s}(t)dt-\partial_zf_n^*(t)x_{\theta,s}(t)dW(t)\\
           &\q-(D_{\theta}\a_n^*(t)\partial_{yy}f_n^*(t)+D_{\theta}\b_n^*(t)\partial_{yz}f_n^*(t))y_s(t)dt\\
           &\q-(D_{\theta}\a_n^*(t)\partial_{yz}f_n^*(t)+D_{\theta}\b_n^*(t)\partial_{zz}f_n^*(t))y_s(t)dW(t),\q   \theta\leq t\leq T, \\
 \displaystyle
 x_{\theta,s}(\theta)&=0,\\
 x_{\theta,s}(t)&=0,\q 0\leq t <\theta\leq T.
\end{split}
\right.
\end{equation*}

For any $p\in [2, \,q)$, by Lemma \ref{luck} and \eqref{wang41}, we have
\begin{equation}\label{bai31}
\begin{aligned}
&\me\sup_{\theta\leq t\leq T}|D_{\theta}(\Psi^*(t)\Phi^*(s))x_0|^p\\
\leq & C\bigg[\me\Big(\int_\theta^T|(D_{\theta}\a_n^*(t)\partial_{yy}f_n^*(t)+D_{\theta}\b_n^*(t)\partial_{yz}f_n^*(t))y_s(t)|dt\Big)^p\\
    &\qq+\me\Big(\int_\theta^T|(D_{\theta}\a_n^*(t)\partial_{yz}f_n^*(t)+D_{\theta}\b_n^*(t)\partial_{zz}f_n^*(t))y_s(t)|^2dt\Big)^{\frac{p}{2}}\bigg]\\
\leq& C  \bigg(\me\sup_{\theta\leq t\leq T}|D_\theta\a_n(t)|^q+\me\Big(\int_\theta^T|D_\theta\b_n(t)|^2dt\Big)^{\frac{q}{2}}\bigg)^{\frac{p}{q}}\big(\me\sup_{s\leq t\leq T}|y_s(t)|^{\frac{pq}{q-p}}\big)^{\frac{q-p}{q}}\\
\leq& C \Big(\sup_\theta\me|D_{\theta}\alpha_n(T)|^q\Big)^{\frac{p}{q}}|x_0|^p\\
\leq & C|x_0|^p,
\end{aligned}
\end{equation}
where $C$ depends only on $q$, $L,\,T$ and $M$. Therefore, by the Clark-Ocone-Haussman formula,
\begin{equation*}\label{bai32}
\begin{aligned}
&\Phi(s)\Psi(T)D_{\theta}\a_n(T)\\
=&\me(\Phi(s)\Psi(T)D_{\theta}\a_n(T))+\int_0^T\me\Big(D_t\big(\Phi(s)\Psi(T)D_{\theta}\a_n(T)\big)\Big|\mf_t\Big)dW(t)\\
=&\me(\Phi(s)\Psi(T)D_{\theta}\a_n(T))\\
 &\q+\int_0^T\me\Big(D_t\big(\Phi(s)\Psi(T)\big)D_{\theta}\a_n(T)+\Phi(s)\Psi(T)D_tD_{\theta}\a_n(T)\Big|\mf_t\Big)dW(t)\\
:=&\me(\Psi(T)D_{\theta}\a_n(T))+\int_0^Tu_{\theta,s}(t)dW(t).
\end{aligned}
\end{equation*}
Fixing $p_0$ with $2<p_0< q/2$, by \textbf{(A2)}, \eqref{wang41} and \eqref{bai31},  we have
\begin{equation*}\label{bai32}
\begin{aligned}
\me|u_{\theta,s}(t)|^{p_0}
=&\me\Big|\me\big(D_t(\Phi(s)\Psi(T))D_{\theta}\a_n(T)+\Phi(s)\Psi(T)D_tD_{\theta}\a_n(T)\big|\mf_t\big)\Big|^{p_0}\\
\leq&C \big(\me|D_t(\Phi(s)\Psi(T))D_{\theta}\a_n(T)|^{p_0}+\me|\Phi(s)\Psi(T)D_tD_{\theta}\a_n(T)|^{p_0}\big)\\
\leq& C\Big(\me|D_t(\Phi(s)\Psi(T))|^{\frac{p_0q}{q-p_0}}\Big)^{\frac{q-p_0}{q}}\Big(\me|D_\theta\a_n(T)|^{q}\Big)^{\frac{p_0}{q}}\\
&\q+C\Big(\me|\Phi(s)\Psi(T)|^{\frac{p_0q}{q-p_0}}\Big)^{\frac{q-p_0}{q}}\Big(\me|D_tD_\theta\a_n(T)|^{q}\Big)^{\frac{p_0}{q}}\\
\leq &C,
\end{aligned}
\end{equation*}
where $C$ depends only on $q$, $L,\,T$ and $M$, which proves $\Phi(s)\Psi(T)D_{\theta}\a_n(T)\in M^{2,p_0}\subset M^{2,2}.$

\end{proof}

Now, we present the backward Euler method for Eq. \eqref{bsden}.
 Suppose a partition $\pi:0=t_0<t_1<\cdots<t_N=T$ of $[0,T]$ with the mesh
 size $\displaystyle|\pi|=\max_{0\leq i\leq N} |t_{i+1}-t_i|$. Then we denote $\D_i=t_{i+1}-t_i$ and $\D_i
 W=W(t_{i+1})-W(t_i)$, for $i = 0, 1, \cdots, N-1$.

For simplicity, we assume that $\D_i=|\pi|=\frac{T}{N}$, for each $i=0,1,\cdots,N-1$. Our numerical scheme
still works for general uniform partition of $[0,T]$ (i.e., there exists a constant $K$, such that $K|\pi|\leq \D_j$, for any $j=0,1,\cdots,N-1$).

Throughout this paper, we assume that $|\pi|\leq 1.$

%
%

 For the partition $\pi$, we introduce the implicit backward Euler method for Eq. \eqref{bsden} as
\begin{equation}\label{nm1}
\begin{aligned}
\alpha_n^\pi(t_{j+1})-\alpha_n^\pi(t)
=&\bigg(\L \alpha_n^\pi(t_j)+f_n\Big(t_j,\alpha_n^\pi(t_j),\frac{1}{\Delta_j}
\me\Big(\int_{t_j}^{t_{j+1}}\beta_n^\pi(s)ds\big|\mf_{t_j}\Big)\Big)\bigg)(t_{j+1}-t)\\
&+\int_{t}^{t_{j+1}}\beta_n^\pi(s)dW(s)\\
 \end{aligned}
 \end{equation}
for any  $j=0,1,\cdots,N-1$ and $\a_{n}^{\pi}(t_N)={\lan q_T^\pi,\phi_i\ran}_{L^2(D)}$,
where $q_T^\pi$ is an approximation of $q_T$.

\begin{remark}
Multiplying both sides of \eqref{nm1} by $\D_jW$, and then taking expectation, we have
\begin{equation*}
\begin{aligned}
\me\int_{t_j}^{t_{j+1}}\beta^\pi_n(s)ds=\me(\alpha^\pi_n(t_{j+1})\,\D_jW).
\end{aligned}
\end{equation*}
Furthermore, 
\begin{equation*}
\begin{aligned}
\alpha^\pi_n(t_{j})=\L_j^{-1}\Big[\me(\alpha^\pi_n(t_{j+1})|\mf_{t_j})-f_n\big(t_j,\alpha^\pi_n(t_{j}),\frac{1}{\D_j}\me(\alpha^\pi_n(t_{j+1})\,\D_jW)\big)\Big],
\end{aligned}
\end{equation*}
where
$$\L_j=\begin{pmatrix}
1+\l_1\D_j& 0  &\cdots &0\\
0   & 1+\l_2\D_j&\cdots&0\\
\vdots&\vdots&\ddots&\vdots\\
0&  0     &\cdots &1+\l_n\D_j\\
\end{pmatrix}.$$
Therefore, this scheme involves the computation of conditional expectations with respect to $\mf_{t_j}$.  In this respect, the Monte-Carlo method is a popular choice. We refer the reader to the related work (\cite{Bender-Denk07,Bouchard-Touzi04,Gobet-Lemor-Warin}) for details.
\end{remark}

The following lemma comes from \cite[Lemma 5.4]{Zhang04}.

\begin{lemma} \label{Gronwall's ineq} %
Suppose that  $a_i\geq 0,b_i\geq 0,c>0, a_{i}\leq(1+c\D_{i})a_{i+1}+b_{i+1}$, for any $i=0,1,\cdots,N-1$.Then
$$\max_{0\leq i\leq n}a_i\leq e^{cT}(a_n+\sum_{i=1}^n b_i).$$

\end{lemma}


The following result is the convergence speed for the backward Euler method indicated in \eqref{nm1}.

\begin{theorem}\label{I_j}
Let $\bf{(A1)}$ and $\bf{(A2)}$ hold, and suppose that $\l_n^2|\pi|\leq 1$ is true. Then
\begin{equation}\label{nmm2}
\begin{aligned}
&\me \max_{0\leq j\leq
N}|\a_{n}(t_j)-\a_{n}^{\pi}(t_j)|^2+\me\int_0^T|\b_{n}(t)-\b_{n}^\pi(t)|^2dt\\
\leq&
C\l_n^2\Big(\me |\a_{n}^{\pi}(T)-\a_{n}(T)|^2+|\pi|\Big),
\end{aligned}
\end{equation}
where  $C$ is a constant depending only on $\l_1,\,q,\,T,\,L$ and $M$.
\end{theorem}%

\begin{proof}
We divide the proof into two steps.

\textbf{Step 1.} 
Under the condition $\l_n^2|\pi|\leq 1$, by Lemma \ref{beta}, we have, for any $t,s \in [0,T]$,
\begin{equation}\label{wang61}
\begin{aligned}
\me|\beta_{n}(t)-\beta_{n}(s)|^2\leq C|t-s|(\l_n^2|t-s|+1),
\end{aligned}
\end{equation}
where $C$ is a constant depending only on $q,\, T$, $L$ and $M$.

 For any
$j=0,1,\cdots,N-1$, taking $t=t_j$ in \eqref{bsden} and \eqref{nm1}, one obtains that
 \begin{equation}\label{I_j 2}
\begin{aligned}
 \displaystyle
&\L_j(\alpha_{n}(t_j)-\a_{n}^\pi(t_j))+\int_{t_j}^{t_{j+1}}
(\beta_{n}(t)-\b_{n}^\pi(t))dW(t)\\
=&(\alpha_{n}(t_{j+1})-\a_{n}^\pi(t_{j+1}))+\int_{t_j}^{t_{j+1}}
\L(\alpha_{n}(t_j)-\alpha_{n}(t))dt \\ &+\int_{t_j}^{t_{j+1}}\bigg(f_n\Big(t_j,\alpha_n^\pi(t_j),\frac{1}{\Delta_j}
\me\big(\int_{t_j}^{t_{j+1}}\beta_n^\pi(s)ds\big|\mf_{t_j}\big)\Big)
-f_n(t,\alpha_n(t),\beta_n(t))\bigg)dt.\\
\end{aligned}
 \end{equation}
Since
\begin{equation*}
\begin{aligned}
&f_n\Big(t_j,\alpha_n^\pi(t_j),\frac{1}{\Delta_j}
\me\Big(\int_{t_j}^{t_{j+1}}\beta_n^\pi(s)ds\big|\mf_{t_j}\Big)\Big)
-f_n(t,\alpha_n(t),\beta_n(t))\\
\leq& L\bigg(\sqrt{t-t_j}+|\alpha_n^\pi(t_j)-\alpha_n(t)|+\Big|\frac{1}{\Delta_j}
\me\Big(\int_{t_j}^{t_{j+1}}\beta_n^\pi(s)ds\big|\mf_{t_j}\Big)-\beta_n(t)\Big|\bigg)\\
\leq & L\bigg(\sqrt{t-t_j}+|\alpha_n^\pi(t_j)-\alpha_n(t_j)|+|\alpha_n(t_j)-\alpha_n(t)|\\
&\q+\Big|\frac{1}{\Delta_j}
\me\Big(\int_{t_j}^{t_{j+1}}\beta_n^\pi(s)
-\beta_n(t_j)ds\big|\mf_{t_j}\Big)\Big|+|\beta_n(t_j)-\beta_n(t)|\bigg)\\
\leq & L\bigg(\sqrt{t-t_j}+|\alpha_n^\pi(t_j)-\alpha_n(t_j)|+|\alpha_n(t_j)-\alpha_n(t)|
    +|\beta_n(t_j)-\beta_n(t)|\\
&\q+\Big|\frac{1}{\Delta_j}
 \me\Big(\int_{t_j}^{t_{j+1}}\beta_n^\pi(s)
   -\beta_n(s)ds\big|\mf_{t_j}\Big)\Big|+\Big|\frac{1}{\Delta_j}
    \me\Big(\int_{t_j}^{t_{j+1}}\beta_n(s)
    -\beta_n(t_j)ds\big|\mf_{t_j}\Big)\Big|\bigg),
\end{aligned}
\end{equation*}
squaring both sides of \eqref{I_j 2} and then taking expectation, we get
 \begin{equation}\label{I_j 2s}
\begin{aligned}
&\me|\L_j(\alpha_{n}(t_j)-\a_{n}^\pi(t_j))|^2+\me\int_{t_j}^{t_{j+1}}
|\beta_{n}(t)-\b_{n}^\pi(t)|^2dt\\
\leq &(1+7\D_j/\e)\me|\alpha_{n}(t_{j+1})-\a_{n}^\pi(t_{j+1})|^2+(7+\e/\D_j)\bigg\{\me\bigg|\int_{t_j}^{t_{j+1}}
\L(\alpha_{n}(t_j)-\alpha_{n}(t))dt\bigg|^2\bigg.\\
&+\frac{4L^2\D_j^3}{9} +L^2\D_j^2\me|\alpha_n^\pi(t_j)-\alpha_n(t_j)|^2+L^2\me\Big(\int_{t_j}^{t_{j+1}}
|\alpha_n(t_j)-\alpha_n(t)|dt\Big)^2\\
&+L^2\me\Big(\int_{t_j}^{t_{j+1}}
|\beta_n(t_j)-\beta_n(t)|dt\Big)^2+L^2\me\bigg|\me\Big(\int_{t_j}^{t_{j+1}}
\beta_n^\pi(t)-\beta_n(t)dt\Big|\mf_{t_j}\Big)\bigg|^2\\
&\bigg.+L^2\me\bigg|\me\Big(\int_{t_j}^{t_{j+1}}
\beta_n(t)-\beta_n(t_j)dt\Big|\mf_{t_j}\Big)\bigg|^2\bigg\}\\
\leq &(1+7\D_j/\e)\me|\alpha_{n}(t_{j+1})-\a_{n}^\pi(t_{j+1})|^2\\
&+(7+\e/\D_j)
\bigg\{(\l_n^2\D_j+L^2\D_j)\me\int_{t_j}^{t_{j+1}}
|\alpha_{n}(t_j)-\alpha_{n}(t)|^2dt +\frac{4L^2\D_j^3}{9}\bigg.\\
&+L^2\D_j^2\me|\alpha_n^\pi(t_j)-\alpha_n(t_j)|^2+2L^2\D_j\me\int_{t_j}^{t_{j+1}}
|\beta_n(t_j)-\beta_n(t)|^2 dt\\
&\bigg.+L^2\D_j\me\int_{t_j}^{t_{j+1}}
|\beta_n^\pi(t)-\beta_n(t)|^2 dt\bigg\}.\\
 \end{aligned}
 \end{equation}
Therefore,
 \begin{equation}\label{I_j 3s1}
\begin{aligned}
&\Big(1+\l_1\D_j-\Big(7+\frac{\e}{\D_j}\Big)L^2\D_j^2\Big)\me|\alpha_{n}(t_j)-\a_{n}^\pi(t_j)|^2\\
&\qq\q+\Big(1-\Big(7+\frac{\e}{\D_j}\Big)L^2\D_j\Big)\me\int_{t_j}^{t_{j+1}}
|\beta_{n}(t)-\b_{n}^\pi(t)|^2dt\\
\leq &(1+7\D_j/\e)\me|\alpha_{n}(t_{j+1})-\a_{n}^\pi(t_{j+1})|^2\\
&+(7+\e/\D_j)
\bigg\{(\l_n^2\D_j+L^2\D_j)\me\int_{t_j}^{t_{j+1}}
|\alpha_{n}(t_j)-\alpha_{n}(t)|^2dt \bigg.\\
&\bigg.+\frac{4L^2\D_j^3}{9}+2L^2\D_j\me\int_{t_j}^{t_{j+1}}
|\beta_n(t_j)-\beta_n(t)|^2 dt\bigg\}.\\
 \end{aligned}
 \end{equation}
Set $\ds\e=\frac{\l_1-7L^2|\pi|}{L^2}$. 
Then 
$$\ds1-\Big(7+\frac{\e}{\D_j}\Big)L^2\D_j^2\geq 1-\l_1\D_j\geq\frac{1}{2} ,\,\,1+\l_1\D_j-\Big(7+\frac{\e}{\D_j}\Big)L^2\D^2_j\geq 1,$$ 
$$ 1+7\D_j/\e=1+\frac{7L^2\D_j}{\l_1-7L^2|\pi|},\,\,
  7+\e/\D_j=7+\frac{\l_1-7L^2|\pi|}{L^2\D_j}.$$
By \eqref{I_j 3s1}, we can easily obtain that
\begin{equation*}\label{I_j 4s}
\begin{aligned}
&\me|\alpha_{n}(t_j)-\a_{n}^\pi(t_j)|^2
+\frac{1}{2}\me\int_{t_j}^{t_{j+1}}
|\beta_{n}(t)-\b_{n}^\pi(t)|^2dt\\
\leq &\big(1+\frac{7L^2\D_j}{\l_1-7L^2|\pi|}\big)\me|\alpha_{n}(t_{j+1})-\a_{n}^\pi(t_{j+1})|^2
+\big(7+\frac{\l_1-7L^2|\pi|}{L^2\D_j}\big)
\bigg\{\frac{4L^2\D_j^3}{9}+ \bigg.\\
&\bigg.(\l_n^2\D_j+L^2\D_j)\me\int_{t_j}^{t_{j+1}}
|\alpha_{n}(t_j)-\alpha_{n}(t)|^2dt+2L^2\D_j\me\int_{t_j}^{t_{j+1}}
|\beta_n(t_j)-\beta_n(t)|^2 dt\bigg\}.\\
 \end{aligned}
 \end{equation*}
Choose $|\pi|$ sufficiently small such that $\l_1-7L^2|\pi|>\l_1/2$. By Lemmas \ref{alpha}-\ref{beta}, we obtain that
\begin{equation}\label{I_j 5s}
\begin{aligned}
&\me|\alpha_{n}(t_j)-\a_{n}^\pi(t_j)|^2
+\frac{1}{2}\me\int_{t_j}^{t_{j+1}}
|\beta_{n}(t)-\b_{n}^\pi(t)|^2dt\\
\leq &(1+C_1\D_j)\me|\alpha_{n}(t_{j+1})-\a_{n}^\pi(t_{j+1})|^2\\
&+\frac{C_2}{\D_j}
\bigg\{\frac{4L^2\D_j^3}{9}+C(\l_n^2\D^3_j+3L^2\D_j^3)(1+\l_n^2|\pi|)\bigg\}\\
\leq&(1+C_1\D_j)\bigg\{\me|\alpha_{n}(t_{j+1})-\a_{n}^\pi(t_{j+1})|^2
+\frac{1}{2}\me\int_{t_{j+1}}^{t_{j+2}}
|\beta_{n}(t)-\b_{n}^\pi(t)|^2dt\bigg\}\\
&+C_2(\l_n^2+L^2+\l_n^4|\pi|)\D_j^2,\\
\end{aligned}
\end{equation}
where $C_1$ and $C_2$ depend  on $\l_1,\,q,\,L,\,T$ and $M$.
Therefore, by Lemma \ref{Gronwall's ineq},
\begin{equation}\label{I_j 8}
\begin{aligned}
 \displaystyle
&\max_{0\leq j\leq N-2}\bigg\{\me|\alpha_{n}(t_j)-\a_{n}^\pi(t_j)|^2
+\frac{1}{2}\me\int_{t_j}^{t_{j+1}}
|\beta_{n}(t)-\b_{n}^\pi(t)|^2dt\bigg\}\\
\leq& e^{C_1T}\bigg(\me|\alpha_{n}(t_{N-1})-\a_{n}^\pi(t_{N-1})|^2
+\frac{1}{2}\me\int_{t_{N-1}}^{t_{N}}
|\beta_{n}(t)-\b_{n}^\pi(t)|^2dt\\
&+C(L^2+\l_n^2+\l_n^4|\pi|)\sum_{j=0}^{N-2}\D_j^2\bigg)\\
\leq&e^{C_1T}\bigg(\me|\alpha_{n}(t_{N-1})-\a_{n}^\pi(t_{N-1})|^2
+\frac{1}{2}\me\int_{t_{N-1}}^{t_{N}}
|\beta_{n}(t)-\b_{n}^\pi(t)|^2dt+C(\l_n^2+\l_i^4|\pi|)|\pi|\bigg).
\end{aligned}
 \end{equation}
Using (\ref{I_j 2}) once more, we have
 \begin{equation}\label{I_j 9}
\begin{aligned}
&\me|\alpha_{n}(t_{N-1})-\a_{n}^\pi(t_{N-1})|^2
+\frac{1}{2}\me\int_{t_{N-1}}^{t_{N}}
|\beta_{n}(t)-\b_{n}^\pi(t)|^2dt\\
\leq&C(\l_n^2+\l_n^4|\pi|)\bigg(\me |\a_{n}^{\pi}(T)-\a_{n}(T)|^2+|\pi|\bigg).
\end{aligned}
 \end{equation}
Combining \eqref{I_j 8} and \eqref{I_j 9}, we obtain that
\begin{equation*}\label{I_j 10}
\begin{aligned}
 \displaystyle
&\max_{0\leq j\leq N-1}\me\bigg\{\alpha_{n}(t_j)-\a_{n}^\pi(t_j)|^2
+\frac{1}{2}\me\int_{t_j}^{t_{j+1}}
|\beta_{n}(t)-\b_{n}^\pi(t)|^2dt\bigg\}\\
\leq&C(\l_n^2+\l_n^4|\pi|)\bigg(\me |\a_{n}^{\pi}(T)-\a_{n}(T)|^2+|\pi|\bigg).
\end{aligned}
 \end{equation*}
\vspace{3mm}

Now, by
(\ref{I_j 5s}), we have
\begin{equation*}
\begin{aligned}
&\displaystyle\sum_{j=0}^{N-2}\bigg\{\me|\alpha_{n}(t_j)-\a_{n}^\pi(t_j)|^2
+\frac{1}{2}\me\int_{t_j}^{t_{j+1}}
|\beta_{n}(t)-\b_{n}^\pi(t)|^2dt\bigg\}\\
\leq
&\sum_{j=0}^{N-2}\bigg\{(1+C_1\D_j)\me|\alpha_{n}(t_{j+1})-\a_{n}^\pi(t_{j+1})|^2
 +C(\l_n^2+\l_n^4|\pi|)\D_j^2\bigg\}.
\end{aligned}
\end{equation*}
Therefore,
\begin{equation}\label{nm4}
\begin{aligned}
&\me\int_0^T|\b_{n}(t)-\b_{n}^\pi(t)|^2dt
=\me\Big(\int_0^{t_{N-1}}+\int_{t_{N-1}}^{t_N}\Big)|\b_{n}(t)-\b^\pi_{n}(t)|^2dt\\
\leq &\sum_{j=0}^{N-2}2C_1\D_{j}\me(\a_{n}(t_{j+1})-\a^\pi_{n}(t_{j+1}))^2
      +(2+2C_1\D_{N-2})\me|\a_{n}(t_{N-1})-\a^\pi_{n}(t_{N-1})|^2\\
    &-2\me|\a_{n}(t_0)-\a^\pi_{n}(t_0)|^2
       +C(\l_n^2+\l_n^4|\pi|)\sum_{j=0}^{N-2}\D_j^2
       +\me\int_{t_{N-1}}^{t_N}|\b_{n}(t)-\b^\pi_{n}(t)|^2dt \\
\leq &\displaystyle C(\l_n^2+\l_n^4|\pi|)\bigg(\me |\a_{n}^{\pi}(T)-\a_{n}(T)|^2+|\pi|\bigg).\\
\end{aligned}
\end{equation}

\textbf{Step 2.}
Similar to \eqref{I_j 2s}, for any $j=0,1,\cdots,N-1$, we have
\begin{equation*}\label{max 1s}
\begin{aligned}
 \displaystyle
&|\L_j(\alpha_{n}(t_j)-\a_{n}^\pi(t_j))|
=|\me(\L_j(\alpha_{n}(t_j)-\a^\pi_{n}(t_j))\,\big|\,\mf_{t_j})|\\
\leq&\me\bigg(|\alpha_{n}(t_{j+1})-\a_{n}^\pi(t_{j+1})|
+\bigg|\int_{t_j}^{t_{j+1}}\L(\alpha_{n}(t_j)-\alpha_{n}(t))dt\bigg|\\
 &+L\int_{t_j}^{t_{j+1}}\bigg(\sqrt{t-t_j}+|\alpha_n^\pi(t_j)-\alpha_n(t_j)|+|\alpha_n(t_j)-\alpha_n(t)|\\
&  +|\beta_n(t_j)-\beta_n(t)|+\Big|\frac{1}{\Delta_j}
 \me\Big(\int_{t_j}^{t_{j+1}}\beta_n^\pi(s)
   -\beta_n(s)ds\big|\mf_{t_j}\Big)\Big|+\Big|\frac{1}{\Delta_j}\\
&    \me\Big(\int_{t_j}^{t_{j+1}}\beta_n(s)
    -\beta_n(t_j)ds\big|\mf_{t_j}\Big)\Big|\bigg)dt\bigg|\mf_{t_j}\bigg)\\
\leq & \me\bigg(|\alpha_{n}(t_{j+1})-\a_{n}^\pi(t_{j+1})|
+\bigg|\int_{t_j}^{t_{j+1}}\L(\alpha_{n}(t_j)-\alpha_{n}(t))dt\bigg|\\
 &+L\D_j|\alpha_n^\pi(t_j)-\alpha_n(t_j)|+L\int_{t_j}^{t_{j+1}}\bigg(\sqrt{t-t_j}+|\alpha_n(t_j)-\alpha_n(t)|\\
&+|\beta_n(t_j)-\beta_n(t)|+\Big|\frac{1}{\Delta_j}
 \me\Big(\int_{t_j}^{t_{j+1}}\beta_n^\pi(s)
   -\beta_n(s)ds\big|\mf_{t_j}\Big)\Big|+\Big|\frac{1}{\Delta_j}\\
&    \me\Big(\int_{t_j}^{t_{j+1}}\beta_n(s)
    -\beta_n(t_j)ds\big|\mf_{t_j}\Big)\Big|\bigg)dt\bigg|\mf_{t_j}\bigg).\\
\end{aligned}
 \end{equation*}
Hence
\begin{equation*}\label{max 2s}
\begin{aligned}
 \displaystyle
&|\alpha_{n}(t_j)-\a_{n}^\pi(t_j)|\\
\leq&\frac{1}{1-L\D_j}\me\bigg(|\alpha_{n}(t_{j+1})-\a_{n}^\pi(t_{j+1})|
+\bigg|\int_{t_j}^{t_{j+1}}\L(\alpha_{n}(t_j)-\alpha_{n}(t))dt\bigg|\\
 &+L\int_{t_j}^{t_{j+1}}\bigg(\sqrt{t-t_j}+|\alpha_n(t_j)-\alpha_n(t)|
    +|\beta_n(t_j)-\beta_n(t)|\\
&+\Big|\frac{1}{\Delta_j}
 \me\Big(\int_{t_j}^{t_{j+1}}\beta_n^\pi(s)
   -\beta_n(s)ds\big|\mf_{t_j}\Big)\Big|+\Big|\frac{1}{\Delta_j}
    \me\Big(\int_{t_j}^{t_{j+1}}\beta_n(s)
    -\beta_n(t_j)ds\big|\mf_{t_j}\Big)\Big|\bigg)dt\bigg|\mf_{t_j}\bigg)\\
\leq &\frac{1}{1-L|\pi|}\me\bigg(|\alpha_{n}(t_{j+1})-\a_{n}^\pi(t_{j+1})|
+b_j\bigg|\mf_{t_j}\bigg)\\
\leq &\Big(\frac{1}{1-L|\pi|}\Big)^2\me\big(|\alpha_{n}(t_{j+2})-\a_{n}^\pi(t_{j+2})|\big|\mf_{t_j}\big)
+\Big(\frac{1}{1-L|\pi|}\Big)^2\me\big(b_{j+1}\big|\mf_{t_j}\big)
+\frac{1}{1-L|\pi|}\me\big(b_j\big|\mf_{t_j}\big)\\
\leq&\cdots\\
\leq&\Big(\frac{1}{1-L|\pi|}\Big)^{N-j}\me\big(|\alpha_{n}(t_{N})-\a_{n}^\pi(t_{N})|\big|\mf_{t_j}\big)
+\sum_{k=1}^{N-j}\Big(\frac{1}{1-L|\pi|}\Big)^k\me\big(b_{j+k-1}\big|\mf_{t_j}\big)\\
\leq&\Big(\frac{1}{1-L|\pi|}\Big)^{N}\me\big(|\alpha_{n}(t_{N})-\a_{n}^\pi(t_{N})|\big|\mf_{t_j}\big)
+\Big(\frac{1}{1-L|\pi|}\Big)^{N}\me\Big(\sum_{k=0}^{N-1}b_{k}\Big|\mf_{t_j}\Big)\\
\leq& C_3\bigg\{\me\big(|\alpha_{n}(t_{N})-\a_{n}^\pi(t_{N})|\big|\mf_{t_j}\big)
+\me\Big(\sum_{k=0}^{N-1}b_{k}\Big|\mf_{t_j}\Big)\bigg\},\\
\end{aligned}
\end{equation*}
where $C_3$ is a constant depending only on $T,\,L$; and 
\begin{equation*}
\begin{aligned}
b_j=&\big|\int_{t_j}^{t_{j+1}}\L(\alpha_{n}(t_j)-\alpha_{n}(t))dt\big|+L\int_{t_j}^{t_{j+1}}\Big(\sqrt{t-t_j}+|\alpha_n(t_j)-\alpha_n(t)|
    +|\beta_n(t_j)-\beta_n(t)|\\
&+\big|\frac{1}{\Delta_j}
 \me\Big(\int_{t_j}^{t_{j+1}}\beta_n^\pi(s)
   -\beta_n(s)ds\big|\mf_{t_j}\Big)\big|+\big|\frac{1}{\Delta_j}
    \me\Big(\int_{t_j}^{t_{j+1}}\beta_n(s)
    -\beta_n(t_j)ds\big|\mf_{t_j}\Big)\big|\Big)dt.
\end{aligned}
\end{equation*}
Therefore,
\begin{equation}\label{max 3s}
\begin{aligned}
 \me\max_{0\leq j\leq N}|\alpha_{n}(t_j)-\a_{n}^\pi(t_j)|^2
\leq C\bigg\{\me|\alpha_{n}(t_{N})-\a_{n}^\pi(t_{N})|^2
+\me\Big(\sum_{k=0}^{N-1}b_{k}\Big)^2\bigg\}.\\
\end{aligned}
\end{equation}
For $\me\big(\sum_{k=0}^{N-1}b_{k}\big)^2$, we have the following estimate:
\begin{equation*}\label{max 4s}
\begin{aligned}
\me\Big(\sum_{k=0}^{N-1}b_{k}\Big)^2
\leq& \me\bigg(\sum_{j=0}^{N-1}\Big( \bigg|\int_{t_j}^{t_{j+1}}\L(\alpha_{n}(t_j)-\alpha_{n}(t))dt\bigg|\\
 &+L\int_{t_j}^{t_{j+1}}\big(\sqrt{t-t_j}+|\alpha_n(t_j)-\alpha_n(t)|
    +|\beta_n(t_j)-\beta_n(t)|\big)dt\\
&+\Big|\me\Big(\int_{t_j}^{t_{j+1}}\beta_n^\pi(s)
   -\beta_n(s)ds\big|\mf_{t_j}\Big)\Big|
   +\Big|\me\Big(\int_{t_j}^{t_{j+1}}\beta_n(s)
    -\beta_n(t_j)ds\big|\mf_{t_j}\Big)\Big|\Big)\bigg)^2\\
\end{aligned}
\end{equation*}
\begin{equation*}
\begin{aligned}
\leq & 6N\sum_{j=0}^{N-1}\Big((\l_n^2\D_j+L^2\D_j)\me\int_{t_j}^{t_{j+1}}|\alpha_n(t_j)-\alpha_n(t)|^2dt
\\
&+2L^2\D_j\me\int_{t_j}^{t_{j+1}}|\beta_n(t_j)-\beta_n(t)|^2dt+L^2\D_j^3 \Big)+6\me\int_0^T|\b_{n}(t)-\b_{n}^\pi(t)|^2dt\\
\leq& C(\l_n^2+\l_n^4|\pi|)\sum_{j=0}^{N-1}\D_j^2+6\me\int_0^T|\b_{n}(t)-\b_{n}^\pi(t)|^2dt\\
\leq &\displaystyle C(\l_n^2+\l_n^4|\pi|)\bigg(\me |\a_{n}^{\pi}(T)-\a_{n}(T)|^2+|\pi|\bigg),\\
\end{aligned}
\end{equation*}
which, together with \eqref{max 3s}, yields that
\begin{equation}\label{max 3s1}
\begin{aligned}
 \me\max_{0\leq j\leq N}|\alpha_{n}(t_j)-\a_{n}^\pi(t_j)|^2
\leq C(\l_n^2+\l_n^4|\pi|)\bigg(\me |\a_{n}^{\pi}(T)-\a_{n}(T)|^2+|\pi|\bigg).\\
\end{aligned}
\end{equation}
From the assumption, $\l_n^2|\pi|\leq 1$ holds. 
The above inequality, together with \eqref{nm4} and \eqref{max 3s1}, yields \eqref{nmm2}.
This completes the proof.
\end{proof}

Finally, set
\begin{equation*}\label{approximation}
\begin{aligned}
q_n^\pi(\cdot)=\sum_{i=1}^n \a_{n,i}^\pi(\cdot)\phi_i,\q
       r_n^\pi(\cdot)=\sum_{i=1}^n \b_{n,i}^\pi(\cdot)\phi_i.
 \end{aligned}
\end{equation*}
 From Theorems \ref{convergence speed} and \ref{I_j},
we obtain the following error estimate for our Galerkin numerical scheme.
\begin{theorem}\label{max convergence}
Let $\bf{(A1)}$ and $\bf{(A2)}$ hold, and suppose that $\l_n^2|\pi|\leq 1$ is true. 
Then
\begin{equation*}\label{max convergence 1}
\begin{aligned}
&\me\max_{0\leq j\leq N}\|q(t_j)-q_n^\pi(t_j)\|_{L^2(D)}^2
 +\me\int_{0}^{T}\|r(t)-r_n^\pi(t)\|_{L^2(D)}^2dt\\
 \leq & C\l_n^2\bigg(\me |\a_{n}^{\pi}(T)-\a_{n}(T)|^2+|\pi|\bigg)+\frac{C}{\l_{n+1}}\[\|q(T)\|^2_{L^2_{\mf_T}(\O;H_0^1(D))}
                        +\|f\|^2_{L^2_{\mathbb{F}}(\O\times(0,T);H_0^1(D))}\].
\end{aligned}
\end{equation*}
Here $C$ depends only on $q,\,T,\,L$, $M$ and $D$, and $\a_{n}^{\pi}(T)={\lan q_T^\pi,\phi_i\ran}_{L^2(D)}$,
where $q_T^\pi$ is an approximation of $q_T$.
\end{theorem}

\begin{remark}

Generally speaking, one of the main difficulties in constructing a numerical scheme for BSDEs is to guarantee the regularity of the second part of the solution. When the terminal condition is a function of some forward diffusion, Ma and Zhang \cite{Ma-Zhang02} obtained the $L^2$-regularity of the solution's second component, which is the key point of Euler method; when terminal condition has no special form, Hu et al. \cite{Hu-Nualart-Song} also obtained the $L^2$-regularity under suitable conditions in terms of Malliavin calculus. In this paper, following some idea from \cite{Hu-Nualart-Song}, we impose assumptions $\bf{(A1)}$ and $\bf{(A2)}$.  

In the aforementioned numerical scheme, we suppose that $f$ in Eq. \eqref{110} is a non-random function.
As a matter of fact, we can deal with the random case under suitable assumptions on $f$. The reader can refer to \cite[Theorems 2.3 and 2.6]{Hu-Nualart-Song} for more details.
\end{remark}

\begin{remark}
In \cite{Grecksch-Kloeden96}, a Galerkin algorithm for the following forward parabolic stochastic equation is considered,
\begin{equation}
\left\{
\begin{array}{lll}
dU(t)=(A U(t)+f(U(t)))dt+g(U(t))dW(t), & \mbox{in}\,\, [0,T], \\
U(0)=U_0.
\end{array}
\right.
\end{equation}
Under suitable assumption of $f,\,g$ and $U_0$, the discretisation error is bounded by
$$C\Big(\frac{1}{\l_{n+1}}+\l_n^{2([r+1/2]+1)}|\pi|^{2r}\Big),$$
where
$2r$ is a positive integer, $[x]$ is the integer part of the real number $x$ and $C$ is a constant depending only on $U_0,\,f,\,g$ and $T$.  To some extent our result
is consistent with that of the forward equations.
\end{remark}

\section*{Acknowledgement}
The author acknowledges gratefully Professor Xu Zhang for his valuable suggestions during this work.

\end{document}